\documentclass[12pt, reqno]{amsart}

\usepackage{amsfonts,latexsym,amsthm,amssymb,amsmath,amscd,euscript,bm}
\usepackage[margin = 2cm]{geometry}
\usepackage{color}
\usepackage{hyperref}
\usepackage{url}
\usepackage{breakurl}
\usepackage{float}
\newcommand{\bburl}[1]{\textcolor{blue}{\url{#1}}}

\newtheorem{thm}{Theorem}[section]
\newtheorem{lem}[thm]{Lemma}

\newtheorem{defi}[thm]{Definition}
\newtheorem{rek}[thm]{Remark}

\newcommand\ben{\begin{enumerate}}
\newcommand\een{\end{enumerate}}

\newcommand{\mattwo}[4]
{\left(\begin{array}{cc}
                        #1  & #2   \\
                        #3 &  #4
                          \end{array}\right) }

\newcommand{\proofcase}{%
  \ifusedcase\else\usedcasetrue\stepcounter{case}\fi
  \par
  \refstepcounter{proofcase}
  \everypar=\expandafter{\the\everypar{\setbox0=\lastbox}\everypar{}Case \theproofcase\ }%
}

\newcommand{\R}{\mathbb{R}}

\newcommand{\gep}{\epsilon}  
\newcommand{\gl}{\lambda}    

\newcommand{\twocase}[5]{#1 \begin{cases} #2 & \text{\rm #3}\\ #4
&\text{\rm #5} \end{cases}  }

\newcommand\be{\begin{equation}}
\newcommand\ee{\end{equation}}
\newcommand\bea{\begin{eqnarray}}
\newcommand\eea{\end{eqnarray}}

\newcommand{\foh}{\frac{1}{2}}  

\newcommand{\hkpn}{H_k^+(N)}
\newcommand{\hkmn}{H_k^-(N)}
\newcommand{\hkpmn}{H_k^\pm(N)}
\newcommand{\hkn}{H_k^\ast(N)}

\newcommand{\phir}[1]{\widehat{\phi_j}\left( \frac{ \log p_{#1} }{\log R}\right) }

\newcommand{\pfrac}[1]{\frac{2\log p_{#1}}{\sqrt{p_{#1}} \log R}}

\newcommand{\hphi}{\widehat{\phi}}  
\newcommand{\phihat}{\widehat{\phi}} 
\newcommand{\hpsi}{\widehat{\psi}}  
\newcommand{\hvarphi}{\widehat{\varphi}}  

\newcommand{\supp}{\operatorname{supp}}
\newcommand{\notdiv}{\nmid}
\newcommand{\intinf}{\int_{-\infty}^\infty}
\newcommand{\I}{1\!\!1} 

\renewcommand{\i}{{\mathrm{i}}} 
\renewcommand{\d}{{\mathrm{d}}} 

\renewcommand{\Re}{{\mathfrak{Re}}}

\newcommand{\<}{\left\langle}
\renewcommand{\>}{\right\rangle}

\newcommand{\nc}{\newcommand}

\nc{\Symp}{\mathsf{Sp}}
\nc{\SpOrthO}{\mathsf{SO(odd)}}
\nc{\SpOrthE}{\mathsf{SO(even)}}
\nc{\Orth}{\mathsf O}
\nc{\Unit}{\mathsf U}
\nc{\UnitSp}{\mathsf{USp}}

\nc{\N}{\mathbb N}
\nc{\DD}{\mathbb D}
\nc{\TT}{\mathbb T}
\nc{\EE}{\mathbb E}

\nc{\cT}{\mathcal T}
\nc{\cP}{\mathcal P}
\nc{\cM}{\mathcal M}
\nc{\cC}{\mathcal C}
\nc{\cB}{\mathcal B}
\nc{\cG}{\mathcal G}
\nc{\cA}{\mathcal A}
\nc{\cS}{\mathcal S}
\nc{\cF}{\mathcal F}
\nc{\cL}{\mathcal L}
\nc{\cR}{\mathcal R}

\numberwithin{equation}{section}
\numberwithin{part}{section}

\title
{
	\textsc{Bounding Vanishing at the Central Point of Cuspidal Newforms}
}

\author[Li]{Jiahui (Stella) Li}
\email{\textcolor{blue}{\href{mailto:stellali110204@gmail.com}{stellali110204@gmail.com}}}
\address{Emma Willard School, Troy, NY 12180}

\author[Miller]{Steven J. Miller}
\email{\textcolor{blue}{\href{mailto:sjm1@williams.edu}{sjm1@williams.edu}},  \textcolor{blue}{\href{Steven.Miller.MC.96@aya.yale.edu}{Steven.Miller.MC.96@aya.yale.edu}}}
\address{Department of Mathematics and Statistics, Williams College, Williamstown, MA 01267}

\date{\today}



\thanks{We thank the participants of the 2019, 2020 and 2021 Williams College SMALL REU Program for helpful comments on earlier drafts, and attendees at the 2021 Recent Developments in Number Theory Conference. We also thank the referee for many helpful comments.}

\begin{document}


\maketitle

\thispagestyle{empty}

\begin{abstract}
The Katz-Sarnak Density Conjecture states that zeros of families of $L$-functions are well-modeled by eigenvalues of random matrix ensembles. For suitably restricted test functions, this correspondence yields upper bounds for the families' order of vanishing at the central point. We generalize results on the $n$\textsuperscript{th} centered moment of the distribution of zeros to arbitrary test functions. On the computational side, we use our improved formulas to obtain significantly better bounds on the order of vanishing for cuspidal newforms, setting records for the quality of the bounds. We also discover better test functions that further optimize our bounds. We see improvement as early as the $5$\textsuperscript{th} order, and our bounds improve rapidly as the rank grows (more than one order of magnitude better for rank 10 and more than four orders of magnitude for rank 50).
\end{abstract}


\section{Introduction}
\subsection{Background}

Since the success of the Riemann zeta function $\zeta(s)$ in converting questions on primes to related ones on integers, a zoo of generalizations have emerged where local data leads to a global object whose properties help resolve questions of interest. Dirichlet $L$-functions, for example, are used to study primes in arithmetic progression. They have an Euler product of degree $1$. With Euler products of degree $2$, cuspidal newforms are the next simplest. One of the best examples of cuspidal newforms is the $L$-function associated with an elliptic curve, which is used to study the Mordell-Weil group (the group of rational solutions of an elliptic curve). 

We briefly describe the main number theory objects of study, cuspidal newforms. For details see \cite{ILS, IK}; for our purposes, all that we need are that these objects are generalizations of the Riemann zeta function and of interest to number theorists. 

Modular forms are functions of a complex variable from the upper half plane to the complex numbers which have decay properties (we will elaborate in detail shortly) and transform nicely under congruence groups, which are subsets of $2 \times 2$ matrices with integer entries and determinant 1. In particular, if $\gamma$ is such a matrix, say $\gamma = \mattwo{a}{b}{c}{d}$ with $ad - bc = 1$, we set $\gamma z = \frac{az + b}{cz  + d}$. A simple calculation shows that if $z$ is in the upper half plane (so ${\rm Im}(z) > 0$), so too is $\gamma z$. Modular forms are generalizations of periodic functions; instead of requiring $f(z+1) = f(z)$ we require $f(\gamma z) = (cz +d)^{2k} f(z)$ (plus boundedness as the imaginary part of $z$ tends to infinity; see \cite{IK} for details). In particular, if we take $\gamma = \mattwo{1}{1}{0}{1}$ we find $f(\gamma z) = f(z+1) = f(z)$ as $cz + d = 1$ here, which explains why these are viewed as generalizations of periodic functions. The difference is that we now require our function to transform well under other shifts. The modular forms of level $N$ transform nicely under such matrices with the additional constraint that $c \equiv 0 \bmod N$; if $N=1$ one can show the set of possible matrices are generated by $\mattwo{1}{1}{0}{1}$, which sends $z$ to $z+1$, and $\mattwo{0}{-1}{1}{0}$, sending $z$ to $-1/z$. Note if a form is of level $N$ it is trivially of form level $M$ for any $M$ divisible by $N$. We say a form is a newform of level $N$ if it is not a form of any smaller level.

We collect these requirements, and some properties, in the definition below. We study cuspidal newforms, which are modular forms that tend to zero as the imaginary part tends to infinity; thus, in the series expansions the constant term is zero. 

\begin{defi}[Cuspidal Newforms]\label{def: Dirichlet L-functions}
Let $H^\star_k(N)$ denote the set of holomorphic cusp forms of weight $k$ that are newforms of level $N$. For every $f\in H^\star_k(N)$, we have a Fourier expansion
\begin{equation}
f(z)\ =\ \sum_{n=1}^\infty a_f(n) e(nz).
\end{equation}
We set $\lambda_f(n) =  a_f(n) n^{-(k-1)/2}$, and obtain the $L$-function associated to $f$ 
\begin{equation}
L(s,f)\ =\ \sum_{n=1}^\infty \lambda_f(n) n^{-s}.\label{eq: L-function}
\end{equation}
The completed $L$-function is
\begin{equation}\label{eq:completed_L_func}
\Lambda(s,f) \ =\ \left(\frac{\sqrt{N}}{2\pi}\right)^s
\Gamma\left(s+\frac{k-1}{2}\right) L(s,f).\end{equation}
Since $\Lambda(s,f)$ satisfies the functional equation $\Lambda(s,f) =
\epsilon_f \Lambda(1-s,f)$ with $\epsilon_f = \pm 1$, $H^\star_k(N)$ splits into two disjoint subsets, $H^+_k(N) = \{
f\in H^\star_k(N): \epsilon_f = +1\}$ and $H^-_k(N) = \{ f\in
H^\star_k(N): \epsilon_f = -1\}$.

The associated symmetry group of $H^\star_k(N)$ is Orthogonal \emph{(O)}, $H^+_k(N)$ is Special Orthogonal even \emph{SO(even)}, and $H^-_k(N)$ is Special Orthogonal odd \emph{SO(odd)}. 

Note that \emph{SO(even)($N$)} is the scaling limit of \emph{SO($2N$)}, and \emph{SO(odd)($N$)} is the scaling limit of \emph{SO($2N+1$)}.
\end{defi}

The Generalized Riemann Hypothesis (GRH) asserts that all non-trivial zeros of $L$-functions are located on the critical line $\Re(s) = 1/2$. The special case for $\zeta(s)$ yields an error term in the prime number theorem roughly on the order of square-root cancellation. More generally, the GRH enables us to investigate many statistics of the normalized zeros as we can now order them on the critical line. One of the first studied is the $n$-level correlation; see \cite{Hej, Mon, RS}.

\begin{defi}[$n$-level Correlation]\label{def: n-level correlation} Given an increasing sequence $\{\alpha_n\}_{n=1}^\infty$ and a box $B$ $\subset \mathbb{R}^{n-1}$, the $n$-level correlation is defined by
\begin{equation}\label{eq: n-level correlation}
    \lim_{N \to \infty} \frac{\#\{(\alpha_{j_1}-\alpha_{j_2},\dots,\alpha_{j_{N-1}}-\alpha_{j_N})\in B, j_i \neq j_k\}}{N}.
\end{equation}
The pair correlation is the case when $n=2$.
\end{defi}

If we know all the $n$-level correlations, through combinatorics and inclusion / exclusion we can pass to the distribution of gaps between adjacent zeros. While these calculations cannot be done in general, they are well-modeled by random matrix theory. That subject began with Wishart's \cite{Wis} work in the 1920's with applications in statistics, and has since been successfully applied in diverse and complex systems such as number theory, nuclear physics, quantum mechanics, and wireless communications; see for example \cite{BFMT-B, Con, FM, Ha}.

In particular, Montgomery's pair correlation conjecture asserts that the pair correlation between pairs of non-trivial zeros of the Riemann zeta function (normalized to have unit average spacing) is the same as the pair correlation of eigenvalues of random matrices in the Gaussian Unitary Ensemble; Dyson noted that the same functional form arises in random matrix theory, and a similar one models energy levels of heavy nuclei. Odlyzko \cite{Odl1, Odl2} studied millions of normalized zeros near the $10^{20\text{th}}$ and the $10^{22\text{nd}}$ zero of $\zeta(s)$, and these tests reflect remarkable agreements: the spacings between normalized zeros of $\zeta(s)$ far up on the critical line appear indistinguishable from those between the eigenvalues of the Gaussian Unitary Ensemble (GUE). However, the $n$-level correlation have significant limitations, which we address in the next subsection with the introduction of another statistic.

\subsection{Level Densities}

The pair correlation shows an astonishing uniformity: we see the same behavior not just for eigenvalues of complex Hermitian matrices but also for real symmetric, and more importantly for number theory, for the classical compact groups (unitary, orthogonal and symplectic). The reason is that the $n$-level correlations only concern the asymptotic behavior of the distribution of zeros and is, therefore, insensitive to the behavior of finitely many zeros. On the number theory side, Rudnick and Sarnak \cite{RS} showed for suitably restricted test functions the $n$-level correlations of automorphic forms are the same, with the universality due to the first two moments of the Satake parameters. Thus, while initially one thought that zeros were well modeled by the Gaussian Unitary Ensemble, their work showed that any of the above matrix families give the same answer in the limit far from the central point. There is thus a need for a statistic which is sensitive to the behavior near $s = 1/2$, which is the most important point to study inside the critical strip (for example, the Birch and Swinnerton-Dyer conjecture asserts that the order of vanishing of the $L$-function here equals the geometric rank of the Mordell-Weil group).

Although the classical compact groups have identical $n$-level correlation over all their eigenvalues as the matrix size tends to infinity, they have distinct behavior of their eigenvalues near 1. At height $T$, the spacing between adjacent zeros of $L$-functions is on the order of $1/\log(T)$, so one $L$-function provides enough zeros for averaging purposes high up on the line. This is very different near the central point, where we have only finitely many zeros and different behavior here is not drowned out by the infinitely many zeros further up. Thus, in order to differentiate among different classical compact groups and study zeros near the central point, we use another statistic which is sensitive to the behavior of zeros near the central point.

Before stating this new statistic, we review two needed concepts. A function $\phi$ is Schwartz if it is infinitely differentiable and it and all of its derivatives decay faster than any polynomial; explicitly, for any non-negative integers $k$ and $m$ there is an integer $C_{k,m}$ such that \begin{equation}
    |\phi^{(k)}(x)| \ \le \ \frac{C_{k,m}}{(1 + x^2)^m}.\end{equation} The Fourier transform of $\phi$, denote $\widehat{\phi}$, is defined by \begin{equation}
    \widehat{\phi}(y) \ := \ \int_{-\infty}^\infty \phi(x) e^{-2\pi i x y} dx.
\end{equation}

In the arguments below we assume that the Generalized Riemann Hypothesis holds for each $L(s, f)$, that is, we can enumerate the non-trivial zeros of cuspidal newforms $L(s, f)$ of level $N$ and weight $k$ (see \cite{IK}) by \[ \rho^{(j)}_f \ = \  \frac12 + i \gamma_f^{(j)} \] for $\gamma_f^{(j)} \in \R$ increasingly ordered and centered about zero. The number of zeros with $|\gamma_f^{(j)}|$ bounded by an absolute large constant is of order $\log c_f$ for some constant $c_f > 1$; this is known as the \emph{analytic conductor}. It is important to note that the quantities below still exist even if GRH fails; in that case, however, as the zeros are no longer lying on a line we lose the interpretation of talking about spacings between adjacent zeros, though our assumptions on the test function ensures that the sum still converges and makes sense. Additionally, if GRH fails, we cannot use the $n$-level statistics to bound the order of vanishing at the central point, as those arguments crucially depend on the zeros lying on the critical line.

\begin{defi}[$n$-Level Density]
The $n$-level density of an $L$-function $L(s,f)$ is defined as
\begin{equation}
    D_n (f; \Phi) \ := \  \sum_{\substack{j_1, \dots, j_n \\ j_i \neq \pm j_k}} \Phi \left( \frac{\log c_f}{2\pi} \gamma_f^{(j_1)}, \dots, \frac{\log c_f}{2\pi} \gamma_f^{(j_n)} \right)	\label{def:density}
\end{equation}
for a \emph{test function} $\Phi: \R^n \to \R$, For many applications we assume $\Phi$ is a non-negative even Schwartz function with compactly supported Fourier transform and $\Phi(0, \dots, 0) > 0$. 

We will refer to the one-level density as $D (f; \Phi)$.
\end{defi}

\begin{rek}
Unlike the $n$-level correlations, the sum \eqref{def:density} is very hard to study for an individual $f$ because by choice of $\Phi$ it essentially captures only a bounded number of zeros. Instead, we study averages over finite subfamilies $\cF (Q) := \{ f \in \cF : c_f \leq Q \}$ (which are parametrized by some quantity such that as that tends to infinity, the size of the subfamily tends to infinity as well), namely
\begin{equation}
    \EE (D_n (f; \Phi), Q) \ := \  \frac{1}{\# \cF (Q)} \sum_{f \in \cF (Q)} D_n(f; \Phi). \label{eq:averagessecond}
\end{equation}
\end{rek}

The Katz-Sarnak density conjecture \cite{KS1,KS2} asserts that the $n$-level density of a family of $L$-function is only dependent on the symmetry group attached to the family, and converges to the $n$-level density of eigenvalues of a classical compact group. Explicitly,  if $\cF$ is a ``good'' family of $L$-functions, then there exists a distribution $W_{n, \cF}$ such that
\begin{eqnarray}
    & & \lim_{Q \to \infty} \EE (D_n (f; \Phi), Q) = \  \frac{1}{\Phi(0, \dots, 0)} \int_{\R^n} \Phi(x_1, \dots, x_n) W_{n, \cF} (x_1, \dots, x_n) dx_1 \cdots dx_n.
\end{eqnarray}

\begin{thm} [Determinant Expansion \cite{KS1}] Let $K(x) = \frac{\sin{\pi y}}{\pi y}$ and $K_{\epsilon}(x) = K(x-y)+\epsilon K(x+y)$. Then the $n$-level densities have the following distinct closed form determinant expansions for each corresponding symmetry group:
\begin{align}
	W_{n, \SpOrthE} (x) 	
		&\ = \  \det \left( K_1 (x_i, x_j) \right)_{i, j \leq n}, \label{eq:nlevelSOeven} \\
	W_{n, \SpOrthO} (x)
		&\ = \  \det \left( K_{-1} (x_i, x_j) \right)_{i, j, \leq n} + \sum_{k  =  1}^n \delta (x_k) \det \left( K_{-1} (x_i, x_j) \right)_{i, j, \neq k},  \label{eq:nlevelSOodd}\\
	W_{n, \Orth} (x)
		&\ = \  \frac12 W_{n, \SpOrthE} (x) + \frac12 W_{n, \SpOrthO} (x), \label{eq:nlevelSpOrthE}\\
	W_{n, \Unit} (x)
		&\ = \  \det \left( K_0 (x_i, x_j) \right)_{i, j, \leq n}, \label{eq:nlevelUni}\\
	W_{n, \Symp} (x)			
		&\ = \ \det \left( K_{-1} (x_i, x_j) \right)_{i, j, \leq n}.\label{eq:nlevelSymp}
\end{align}
\end{thm}

In addition to distinguishing families of $L$-functions, which the $n$-level correlations cannot do, we can use this statistic to obtain bounds for the order of vanishing of a family of $L$-functions at the central point by choosing a test function which is non-negative and positive at 0.

Our goal is to derive theoretical results for $n$-level densities and related statistics (such as the $n$\textsuperscript{th} centered moment), and then apply these to bound vanishing at the central point. Our bounds depend on the choice of test function and thus a large amount of our work is searching for the best choices. There have been several papers trying to find the optimal test function for given support. To date most of these have focused on using the 1-level density \cite{FrMil, ILS}, though there has been some recent work using the 2-level density \cite{BCDMZ}. For many applications, decent bounds are easily obtained with the following test function, which is non-negative and has bounded support for its Fourier transform (by rescaling our test function, we can always adjust the size of the support).

\begin{defi}[Naive Test Function]\label{def: naive test function} We consider the following Fourier transform pair
\begin{align}
   &  \varphi_{\rm naive}(x)\ =\ \left(\frac{\sin{\pi v x}}{\pi v x}\right)^2, \ \ \ \ \ \nonumber\\
   & \hvarphi_{\rm naive}(y)\ =\ \frac{1}{v}\left(1-\frac{|y|}{v}\right) \ \ \ {\rm if}\ |y| < v \ {\rm and}\ 0 \ {\rm otherwise}.
\end{align}
\end{defi}

\begin{rek}
    Taking $\Phi(x_1, \dots, x_n) = \varphi_{\rm naive}(x_1)\varphi_{\rm naive}(x_2)\cdots \varphi_{\rm naive}(x_n)$ yields decent bounds. While it has been shown in Appendix A of \cite{ILS} and \cite{BCDMZ} that there are other test functions that yield better results than those obtained by using the naive test function, the naive test function for the 1- and 2-level densities is close to the optimal choice, and thus the improvement is small. This is the motivation for our exploration of bounds arising from the higher centered moments, as a new idea is needed to make significant progress.
\end{rek}

\subsection{Centered Moments}
In \cite{ILS} the determinant expansion for the $1$-level density was used to bound the order of vanishing. As the determinant expansion is difficult to work with for higher levels, as computing its Fourier transform becomes complicated, we use an alternative statistic which is more amenable to computations. Our starting point is the $n$\textsuperscript{th} centered moments, studied by Hughes and Miller \cite{HuM}. These provide a more tractable approach to obtaining bounds from higher levels than the determinant expansion from the $n$-level density.\footnote{The Katz-Sarnak density conjectures applies equally to the $n$-level densities or the $n$\textsuperscript{th} centered moments. Thus for any even Schwartz test function, the $n$\textsuperscript{th} centered moments for a family of $L$-functions should agree with the corresponding statistic for eigenangles of a classical compact group.} The following results were proved in \cite{HuM}.

\begin{thm}\label{thm:special case moments}
Let $n \geq 2$, $\supp(\hphi) \subset (-\frac1{n-1},
\frac1{n-1})$, $D(f;\phi)$ be as in \eqref{def:density}, and
define
\begin{align}\label{eq:defRnn-1b}
\nonumber R_n(\phi) \ &= \ (-1)^{n-1} 2^{n-1}\left[ \int_{-\infty}^\infty
\phi(x)^{n}
\frac{\sin 2\pi x}{2\pi x} \;\d x - \foh \phi(0)^{n} \right]  \ = \ 2\int_{-1}^{1} |y|
\hphi(y)^2\;\d y.\\
\sigma^2_\phi  \ &= \ 2\int_{-1}^{1} |y|
\hphi(y)^2\;\d y.
\end{align}
We assume GRH for $L(s,f)$ and for all Dirichlet $L$-functions and send $N\to\infty$ through the primes\footnote{Recall the double factorial means the product of every other integer down to 2 or 1; thus $5!!$ is $5 \cdot 3 \cdot 1$ and $6!!$ is $6 \cdot 4 \cdot 2$.}:
\small
\begin{eqnarray}\label{eq:thmextmomcompactfirst}
\nonumber
& & \lim_{\substack{N\to\infty \\ N \text{\rm prime}}} \<
\left(D(f;\phi) - \< D(f;\phi) \>_\pm \right)^{n}\>_\pm =\
\begin{cases}
(2m-1)!!\ \sigma^{2m}_\phi \pm R_{2m}(\phi) &
\text{\rm if $n=2m$ is even,}\\
\pm R_{2m+1}(\phi) & \text{\rm if $n=2m+1$ is odd.}
\end{cases}
\end{eqnarray}
\end{thm}

\begin{thm}\label{thm:less support}
    Let $n \geq 2$, if $\supp(\hphi) \subset
(-\frac{1}{n}\frac{2k-1}{k}, \frac{1}{n}\frac{2k-1}{k})$, then for the unsplit group,
\small
\begin{eqnarray} \label{eq:mock Gaussian for full group}
\nonumber
& & \lim_{\substack{N\to\infty \\ N\ {\rm prime}}}
\left\langle\left(D(f;\phi)-\<D(f;\phi)\>_\ast\right)^n\right\rangle_\ast =\
\begin{cases}
(2m-1)!! \left(2\intinf \hphi(y)^2 |y| \; \d y\right)^m & \text{\rm if $n=2m$ is even,}\\
0 & \text{\rm if $n = 2m+1$ is odd.}
\end{cases}
\end{eqnarray}
\text{For the split groups, }
\small
\begin{eqnarray}\label{eq:mock Gaussian for split group}
\nonumber
& & \lim_{\substack{N\to\infty \\ N\ {\rm prime}}}
\left\langle\left(D(f;\phi)-\<D(f;\phi)\>_\pm\right)^n\right\rangle_\pm =\
\begin{cases}
(2m-1)!! \left(2\intinf \hphi(y)^2 |y| \; \d y\right)^m & \text{\rm if $n=2m$ is even,}\\
0 & \text{\rm if $n = 2m+1$ is odd.}
\end{cases}
\end{eqnarray}

\end{thm}

The advantage of the Hughes-Miller results is that we replace the $n \times n$ determinant expansions with a one-dimensional integral; this is accomplished through a clever change of variable reducing an $n$-dimensional integral involving Bessel functions and $n$ test functions to a related one-dimensional integral with just one Bessel function against a new test function, which is the convolution of the $n$ initial test functions. This is done to take advantage of the results from \cite{ILS}, which handled the 1-dimensional case. The resulting integral can be easily computed, but there is a cost: it no longer resembles an $n$-dimensional quantity and delicate combinatorics are needed to show agreement with the answers from random matrix theory. Overall, this is a good trade, as the new result is more tractable for computations.\footnote{In particular, we can see the emergence of new terms every time the support exceed $(-1/(n-\ell), 1/(n-\ell))$ for any integer $\ell$; see \cite{HuM} for details.}

Theorems \ref{thm:special case moments} and \ref{thm:less support} concern only the special case where all test functions are identical. Thus, compared to the $n$-level density, there is a loss when using the results of Hughes-Miller as we are restricted in what test functions we can use, which decreases the space of functions we can search for optimal test functions. Our main theoretical result, described below, is a generalization of these results to arbitrary test functions. By removing the condition that the test functions must all be equal, we are able to obtain better bounds on vanishing.

\subsection{Main Results}
Miller \cite{Mil1, Mil2} noticed that as $n$ increases, the $n$-level densities provide better and better estimates for bounding the order of vanishing at the central point. We obtain similar bounds from the $n$\textsuperscript{th} centered moments. For each $n$, there is a constant $c_n$ such that the probability of vanishing to order $r$ is at most $c_n / r^n$. For fixed $n$, as $r$ increases we see larger $n$ provided better bounds than those from smaller $r$; however, as the constant $c_n$ grows (its growth comes from the decrease in support of the test functions where the $n$-level density is known) for small $r$ the bounds can actually be worse, or even useless.\footnote{For small $n$ and $r$ the bounds exceed 1, and there are easier ways to show at most 100\% of the forms vanish to a given order!}

Moreover, the results for the $n$\textsuperscript{th} centered moment from \cite{HuM} (Theorem  \ref{thm:special case moments} and Theorem \ref{thm:less support}) lose flexibility because \cite{HuM} only worked with the special case where all test functions are assumed to be identical; in particular $\Phi(x_1,\dots, x_n) = \phi(x_1) \cdots \phi(x_n)$. Thus, on the theory side, our first results is to generalize Theorem  \ref{thm:special case moments} and Theorem \ref{thm:less support} to allow $\Phi$ to be the product of arbitrary test functions that are not necessarily the same.

\begin{thm}\label{thm: generalized centered moment} \emph{(Generalized Theorem \ref{thm:special case moments})}
    Let $n \geq 2$, $\supp(\hphi_j) \subset (-\frac1{n-1},
    \frac1{n-1})$, $D(f;\phi_j)$ be as in \eqref{def:density}, and
    define
    \begin{align}
    R_n(\phi_1,\phi_2,\dots,\phi_n) \ &= \ (-1)^{n-1} 2^{n-1}\left[ \int_{-\infty}^\infty
    \phi_1(x)\cdots \phi_n(x)
    \frac{\sin 2\pi x}{2\pi x} \;\d x \right. \left.- \foh \phi_1(0)\cdots \phi_n(0) \right]\nonumber\\
    \sigma^2_{\phi_j\phi_k} \ &= \ 2\int_{-1}^{1} |y|
    \hphi_j(y)\hphi_k(y)\;\d y.
    \end{align}
    Assume GRH for $L(s,f)$ and for all Dirichlet $L$-functions. As
    $N \to \infty$ through the primes,
    \begin{eqnarray}\label{eq:thmextmomcompactsecond}
    & & \lim_{\substack{N\to\infty \\ N \text{\rm prime}}} \< \prod_{j=1}^n
    \left(D(f;\phi_j) - \< D(f;\phi_j) \>_\pm \right)\>_\pm
    \nonumber \\
    & & =\
    \begin{cases}
    \sum_{k = 1}^{(2m-1)!!} \prod_{\substack{\{a_l,b_l\} \in \mathcal{M}_{2m}^{k} \\ l = 1}}^{m} \sigma^{2}_{\phi_{a_l} \phi_{b_l}} \pm R_{2m}(\phi_1,\dots,\phi_{2m}) &
    \text{\rm if $n=2m$ is even,}\nonumber\\
    \pm R_{2m+1}(\phi_1,\dots,\phi_{2m+1}) & \text{\rm if $n=2m+1$ is odd.} \nonumber
    \end{cases}
    \end{eqnarray}
\end{thm}

\begin{thm}\label{thm: generalized less support}\emph{(Generalized Theorem \ref{thm:less support})} Let $n \geq 2$, if
$\supp(\hphi_j) \subset
(-\frac{1}{n}\frac{2k-1}{k}, \frac{1}{n}\frac{2k-1}{k})$, then for the unsplit group
\begin{eqnarray}\label{eq:generalized mock Gaussian for full group}
    & & \lim_{N\to\infty \atop N \text{\rm prime}} \< \prod_{j=1}^n
    \left(D(f;\phi_j) - \< D(f;\phi_j) \>_\ast \right)\>_\ast  =\
    \begin{cases}
    \sum_{ k = 1}^{(2m-1)!!} \prod_{\substack{\{a_l,b_l\} \in \mathcal{M}_{2m}^{k} \\ l = 1}}^{m} \sigma^{2}_{\phi_{a_l} \phi_{b_l}} &
    \text{\rm if $n=2m$ is even,}\nonumber\\
    0 & \text{\rm if $n$ is odd,}
    \end{cases}
    \end{eqnarray}
    \text{and for the split groups}
    \begin{eqnarray}\label{eq:generalized mock Gaussian for split group}
    & & \lim_{N\to\infty \atop  N \text{\rm prime}} \< \prod_{j=1}^n
    \left(D(f;\phi_j) - \< D(f;\phi_j) \>_\pm \right)\>_\pm  =\
    \begin{cases}
    \sum_{ k = 1}^{(2m-1)!!} \prod_{\{a_l,b_l\} \in \mathcal{M}_{2m}^{k} \atop l = 1}^{m} \sigma^{2}_{\phi_{a_l} \phi_{b_l}} &
    \text{\rm if $n=2m$ is even,}\nonumber\\
    0 & \text{\rm if $n$ is odd.}
    \end{cases}
\end{eqnarray}
\end{thm}

On our calculation side, we are now able to compute bounds for the order vanishing for SO(even) and SO(odd) using the fourth centered moment. We also enumerate bounds from $1$- and $2$-level densities for comparison. We bound the order of vanishing at the central point; many conjectures suggest that for the full family of cuspidal newforms half the forms are of rank 0 and half are of rank 1; however, for sub-families different behavior is possible (in particular, it is possible to construct families of elliptic curves with rank up to 14 at the central point, if we assume the Birch and Swinnerton-Dyer Conjecture, while unconditionally we can build either even or odd signed families that have higher average ranks than the conjectured result for the full family). While we apply our analysis to the case of the split or unsplit families, the test functions we use can yield results in other cases of interest as well. We first define the quantity we wish to bound, the percentage of forms vanishing to a given degree.

\begin{defi}[Order of vanishing / rank] If the series expansion of $L(s,f)$ begins $\alpha_m (s - 1/2)^m + \alpha_{m+1} (s - 1/2)^{m+1} + \cdots$ with $\alpha_m \neq 0$, we say the \emph{order of vanishing} (or \emph{rank}) of the form at the central point is $m$. We let $p_m(N)$ denote the percentage of the forms in our family of level $N$ (either split or unsplit, depending on context) that vanish to order exactly $m$ at the central point $s= 1/2$.
\end{defi}

We use the below inequalities to calculate these bounds. These follow from our expansions as our test functions are even and non-negative; thus if we consider \emph{only} the zeros at the central point, we decrease the sum over the zeros, and thus the integrals provide upper bounds for the action at $s = 1/2$. The contribution from the forms with exactly $m$ zeros at the central point is just the percentage of such forms, namely $p_m(N)$, times the contribution at the central point. As all the zeros contribute equally here, namely $\phi(0)$ or $\phi(0,0)$ and so on, we can move that over to the integral side. Thus, the quantities of interest are now isolated on the right hand side and are independent of the test function, which only appears on the left side. Further, note that if we rescale the test function by $c$ there is no change on the left side, as the integral increases by a factor of $c$ which is then cancelled by the $c$ in the denominator from dividing by the test function.
\ \\
\noindent \textbf{In all the arguments below, we assume the test functions are even and non-negative, and that GRH holds for all $L$-functions under consideration.}

\begin{rek}\label{rek: restriction on k}
     The inequalities for centered moments are only true when we evaluate a certain order of vanishing for $L(s,f)$. Specifically, we need the order of vanishing $r$ to satisfy $r\phi(0)$ $>$ $\hphi(0)+\phi(0)/2$ so that we can maintain the inequality. We expand more on this in Section \ref{subsec: bounds from centered moments sum}.
\end{rek}
\noindent \textbf{For the $1$-level density: }
\begin{align}\label{eq: inequality 1-level density}
	& \frac{1}{\phi(0)} \int_\R \phi(x) W_{1, G} (x) dx  \ \geq \ \ \ \sum_{j = 1}^\infty j p_j(N).
\end{align}
\noindent \textbf{For the $2$-level density: }
\begin{align}\label{eq: inequality 2-level density}
    & \frac{1}{\Phi(0,0)} \int_{\R^2} \Phi(x, y) W_{2, G} (x, y) dx dy \ \geq \ \ \ \sum_{j = 1}^\infty \left(2j(2j - 2) p_{2j}(N) + (2j)^2 p_{2j+1}(N) \right).
\end{align}

\ \\
\noindent \textbf{For even centered moments:}
For the special case where all test functions are equal, we obtain the following:
\begin{align}\label{eq: inequality fourth centered moment}
    & \left\langle\left(  D(f;\phi)-\<D(f;\phi)\>_\sigma\right)^{2m}\right\rangle_\sigma \ \geq \ \ \ \sum_{j=r}^{\infty} p_j(N)\left(r\phi(0)-\left(\hphi(0)+\frac{1}{2}\phi(0)\right)\right)^{2m}.
\end{align}

For the more generalized case, we still have to pair test functions up to ensure that the contribution from each $D(f; \phi_j)$ for each $f$ in the family is positive, so that when we drop terms, the inequality is maintained. Thus, if we are evaluating the $2m$\textsuperscript{th} centered moment, we can pick $m$ test functions arbitrarily, so that the terms are squared. We have the following inequality:

\begin{align} \label{eq: inequality centered moment generalized}
    & \left\langle \prod_{s=1}^{m}\left( D(f;\phi_s)-\<D(f;\phi_s)\>_\sigma\right)^{2}\right\rangle_\sigma\nonumber \geq \ \ \ \sum_{j=r}^{\infty}  p_j(N)\prod_{s=1}^m\left( r\phi_s(0)-\left(\hphi_s(0)+\frac{1}{2}\phi_s(0)\right)\right)^{2}.
\end{align}

We compare our results with the bounds obtained using the $1$- and $2$-level densities, and see \emph{\textbf{significant}} improvements though due to restriction mentioned in Remark \ref{rek: restriction on k}, we could not compute bounds for small vanishing (the exact vanishing restriction depends on the test function we choose) using the centered moment. We also find a relatively good choice for the test function for the $4$\textsuperscript{th} centered moment that yields a better result than $4$ copies of the naive test function.

\begin{table}[H]
    \centering
    \begin{tabular}{|l|l|l|l|}
    \hline
        Order vanishing & $1$-level & $2$-level & $4$\textsuperscript{th} centered moment\textsuperscript{*} \\ \hline
        5  & 0.222908 & 0.0674429 & 0.06580440 \\ \hline
        6  & 0.144090 & 0.0157687 & 0.00853841 \\ \hline
        7  & 0.159220 & 0.0299746 & 0.00221997 \\ \hline
        8  & 0.108067 & 0.0078843 & 0.00081336 \\ \hline
        9  & 0.123838 & 0.0168607 & 0.00036405 \\ \hline
        10 & 0.086454 & 0.0047306 & 0.00018684 \\ \hline
    \end{tabular}\small
        \caption{Comparison of order of vanishing bounds from various approaches. \\
        These are upper bounds for vanishing at least $r$ (number in order vanishing column).\\
        For the $1$-level column, we calculated the bound using the optimal $1$-level bound from \cite{ILS}. The support of the Fourier transform of the test function used is $(-2,2)$.\\
        For the $2$-level column, we calculated the bound using the optimal $2$-level bound from \cite{BCDMZ}. The support of the Fourier transform of the test functions used is $(-1,1)$.\\
        For the $4$\textsuperscript{th} centered moment\textsuperscript{*} column, the \textsuperscript{*} signifies that we used the $4$ copies of the naive test functions $\varphi_{\rm naive}$. The support of the Fourier transform of the test function used is $(-1/3,1/3)$.}
\end{table} \normalsize

The more detailed tables of bounds can be found in \S\ref{sec: table of numerical bounds}.

The paper is organized as follows. In \S\ref{sec: preliminaies} we review number theory and statistic definitions and introduce notations needed for later arguments, and in \S\ref{sec: proof of first theorem} we prove Theorem \ref{thm: generalized centered moment}. In \S\ref{sec: proof for second theorem (less support)} we prove Theorem \ref{thm: generalized less support}, then in \S\ref{sec: bounds order of vanishing} we show how we obtained the inequalities \eqref{eq: inequality 1-level density}, \eqref{eq: inequality 2-level density}, and \eqref{eq: inequality fourth centered moment}. In \S\ref{sec: table of numerical bounds} we explicitly state the significantly better bounds we obtained.

\section{Preliminaries}\label{sec: preliminaies}

In this section, we introduce the various quantities that play key roles in our proofs.

\subsection{Number Theory Preliminaries}

The first item we need is the characteristic function of a set $A$, which detects whether or not $x$ is in $A$.

\begin{defi}[Characteristic Function]
For $A \subset \R$, let  \begin{equation} \twocase{\I_{\{x \in
A\}} \ := \ }{1}{if $x\in A$}{0}{otherwise.}
\end{equation}
\end{defi}

There are multiple definitions of the Fourier Transform; we use the following.

\begin{defi}[Fourier Transform] Given a function $\phi$, its Fourier Transform, $\hphi$, is
\begin{equation}
\nonumber \hphi(y)\ :=\ \intinf \phi(x) e^{-2\pi\i xy} \;\d x, \ \ \ \ \
\phi(x)\ =\ \intinf \hphi(y) e^{2\pi\i xy} \;\d y.
\end{equation}
\end{defi}

Note if $\phi$ is even then the Fourier transform of the Fourier transform is the original function.

In analyzing lower order / error terms, we frequently use big-Oh notation.

\begin{defi} [Big-Oh Notation] A function $f(x)$ is big-Oh of $g(x)$, written $f(x) = O(g(x))$ or $f(x) \ll g(x)$, if there exists an $x_0$ and a $C > 0$ such that for all $x > x_0$ we have $|f(x)| \le C g(x)$. In other words, for $x$ sufficiently large, the absolute value of $f(x)$ is at most a constant times that of $g(x)$.
\end{defi}

In determining the optimal test function we frequently take the convolution of two functions. This process has numerous applications (if $X$ and $Y$ are independent random variables with densities $f$ and $g$ then the convolution of $f$ and $g$ is the density of $X+Y$).

\begin{defi}[Convolution]
Let $f$ and $g$ be two integrable functions. Then their convolution, denoted $f \ast g$, is \begin{eqnarray}
    (f \ast g)(x) \ := \ \int_{-\infty}^\infty f(t) g(x-t) dt.
\end{eqnarray}
\end{defi}

If the functions are square-integrable, by the Cauchy-Schwartz inequality the convolution is finite for each $x$; it is finite for almost all $x$ since the two functions are integrable (integrate the convolution over $x$ and examine the double integral, which splits).

\subsection{Statistics Preliminaries}

We review some statistical terms and concepts that appear in our analysis. The first is needed in the generalization from Hughes-Miller \cite{HuM}, which considered the case where all the test functions were the same, to our more general case of $\Phi(x_1, \dots, x_n) = \phi_1(x_1) \cdots \phi_n(x_n)$. In Hughes-Miller \cite{HuM} we had a combinatorial factor of $(2m-1)!!$, the number of ways to match $2m$ objects in pairs with order not mattering. For us we need not just the number of such pairings, but also which items are paired with which.

\begin{defi}[Set of Sets of Unordered Pairs] We have
$\mathbb{M}_{2m} = \{\{\{a_1,b_1\},\{a_2,b_2\}, \dots, \{a_{m},b_{m}\}\} : \{a_1,b_1,\dots,a_{m},b_{m}\} = \{1,2,\dots,2m\}\}.$
We denote the $k$\textsuperscript{{\rm th}} element of $\mathbb{M}_{2m}$ as $\mathcal{M}_{2m}^k$.
\end{defi}

If we restrict to the special case when all the test functions are the same, we regain the result of Hughes-Miller \cite{HuM} and their combinatorial factor of $(2m-1)!!$ as \begin{equation}   \lvert \mathbb{M}_{2m} \rvert \ =\ \frac{1}{m!}\binom{2m}{2}\binom{2m-2}{2}\cdots\binom{2}{2}\ = \ \ \frac{(2m)!}{2^m m!} \ = \ (2m-1)!!.\nonumber
\end{equation}

\begin{defi}[Statistical Terms]\ \\
\textbf{Expected Value:} For a continuous univariate probability distribution with probability density function $p(x)$, the expected value $\mu$ is
\begin{align*}
    \mu  \ = \ \EE{[X]} \ := \  \int_{-\infty}^{\infty} x p(x) dx.
\end{align*}
Another notation for expected value is $\langle X\rangle$. We use these three notations interchangeably.
\\
\textbf{$n$\textsuperscript{th} Moment:} For a continuous univariate probability distribution with probability density function $p(x)$, the $n$\textsuperscript{{\rm th}} moment $\mu'_n$ is
\begin{align*}
    \mu'_{n}  \ = \ \EE{[X^n]} \ := \ \int _{-\infty }^{+\infty }x^{n}p(x) dx .
\end{align*}
\\
\ \\
\textbf{$n$\textsuperscript{th} Centered Moment:} For a continuous univariate probability distribution with probability density function $p(x)$, the $n$\textsuperscript{{\rm th}} central moment $\mu_n$ is
\begin{align*}
    \mu_n  \ =\ \EE{[(X-\EE{[X]})^n]} \ := \ \int_{-\infty}^{+\infty}(x - \mu)^k p(x) dx.
\end{align*}
\end{defi}

We use the following standard result to convert the $n$\textsuperscript{th} centered moment to the $n$-level density when $n = 2$ and $n = 3$, respectively.
\begin{itemize}
\item The second centered moment $\mu_2$ satisfies $\mu_2 = \mu'_2-\mu^2$.
\item The third centered moment $\mu_3$ satisfies $\mu_3 = \mu'_3-3\mu\mu'_2+2\mu^3$.
\end{itemize}

The next result is used to obtain bounds in our case when the $n$ test functions may differ by replacing each with the same function, which is larger than each. This can only increase our error, and reduces the analysis of certain error terms to the case of Hughes-Miller \cite{HuM} when all test functions are the same. Finding such a function is one of the key new elements needed in generalizing the work from Hughes-Miller.

\begin{lem} \label{lem: functionBound}
    For all real-valued functions $f(x)$, $\lvert f(x) \rvert < f(x)^2+1$.
\end{lem}

\begin{proof} When $|f(x)| \le 1$, the bound clearly holds, while when $|f(x)| > 1$ it is less than its own square. The claim follows by summing the two bounds.
\end{proof}

\subsection{Generalized Moment Sums}
For $D(f; \phi)$ as defined in \eqref{def:density} for the case $n=1$, the explicit formula applied to $D(f;\phi)$ gives (see
Equation 4.25 of \cite{ILS})
\begin{equation}\label{eqdfphiexpansion}
D(f;\phi)\ =\ \widehat{\phi}(0) + \foh \phi(0) - P(f;\phi) +
O\left( \frac{\log \log R}{\log R} \right),
\end{equation}
where
\begin{equation}\label{eq:P in terms of gl}
P(f;\phi)\ =\ \sum_{p \notdiv N} \gl_f(p) \hphi\left(\frac{\log p}{\log R}\right) \pfrac{}.
\end{equation}

When ${\rm supp}(\hphi) \subset (-1,1)$, a similar calculation as the one in \cite{ILS} shows that the $P(f,\phi)$ term does not contribute. Therefore,
\begin{equation}\lim_{N\to\infty} \<D(f;\phi)\>_\sigma \ =\  \hphi(0)+\foh \phi(0) \end{equation} for
any $\sigma \in \{+,-,\ast\}$ (in other words, for either the family with all even signs, all odd signs, or even and odd signs). Thus, to study the $n$\textsuperscript{th} centered moments with $n$ possibly distinct test functions $\phi_1,\phi_2, \dots, \phi_n$, we must evaluate
\begin{align}
\left\langle\prod_{j = 1}^n\left( D(f;\phi_j)-\<D(f;\phi_j)\>_\sigma\right)\right\rangle_\sigma
&\ = \ \left\langle \prod_{j=1}^n\left(-P(f;\phi_j) +
\ O\left( \frac{\log \log R}{\log R}\right)\right)\right\rangle_\sigma \nonumber\\
&\ = \ (-1)^n\left\langle \prod_{j=1}^n P(f;\phi_j)\right\rangle_\sigma \ + \ O\left( \frac{\log \log
R}{\log R} \right). \label{eq:centered mmts of D in terms of P}
\end{align}

The last line follows from Holder's inequality and
$\left\langle \prod_{j=1}^n P(f;\phi_j)\right\rangle_\sigma$ $\ll$ $1$ (which follows from
\eqref{eqdfphiexpansion} and $\<|D(f;\phi_j)|\>_\sigma\ll 1$).

We want to split by the sign of the functional equation and study the families with just even or just odd sign. This is significantly harder than studying the entire family of cuspidal newforms, as the sub-families with even and odd sign have terms that have contributions of the same magnitude but opposite sign; thus in the unsplit family these pieces cancel each other out, and do not need to be analyzed. As we have excellent averaging formulas for sums over the entire family (i.e., the unsplit case), we wish to convert sums over just the even or just the odd terms to related sums over everything. This can easily be done for the space of cuspidal newforms, as there is a good formula for the sign of the functional equation (see equation (3.5) of \cite{ILS}):
If $f\in\hkn$ and $N$ is prime, then
\begin{equation}\label{eq:signfneqexpils}
\gep_f \ = \ -\i^k \gl_f(N) \sqrt{N}.
\end{equation}
As $\gep_f = \pm 1$, \eqref{eq:signfneqexpils} implies $|\gl_f(N)| =
1/\sqrt{N}$. The multiplicative properties of the Fourier coefficients will be essential in our analysis.

Thus, with $\epsilon_f$ the sign of the functional equation for $L(s,f)$ (so it is 1 if the form is even and -1 if the form is odd), we have $(1+\epsilon_f)/2$ is 1 if $f$ is even and 0 otherwise; thus the presence of this factor restricts a sum over all forms to just a sum over the even terms, but does so by replacing a sum over just the even terms with two sums over all forms. Similarly if we instead use $(1 - \epsilon_f)/2$ we would have a sum over just the odd forms.

Using Lemma (2.7) from \cite{HuM} we obtain
\begin{align}
\sum_{f\in\hkpmn} \prod_{j=1}^n P(f;\phi_j) &\ = \
\sum_{f\in\hkn} \frac{1\pm \gep_f}{2} \prod_{j=1}^n P(f;\phi_j) \nonumber\\
& \ = \ \frac12 \sum_{f\in\hkn}\prod_{j=1}^n P(f;\phi_j)\ \mp\ \frac{1}{2}
\sum_{f\in\hkn} \i^k \sqrt{N}\gl_f(N)\prod_{j=1}^n P(f;\phi_j).
\end{align}

From \cite{HuM}, we know that $|\hkpn| \sim |\hkmn| \sim \frac12 |\hkn|$ as $N\to\infty$. Thus, we split by sign and obtain
\begin{equation}\label{eq:splitting}
\left\langle \prod_{j = 1}^nP(f;\phi_j)\right\rangle_\pm\ \sim\ \left\langle \prod_{j = 1}^nP(f;\phi_j)\rangle_\ast\ \mp\ \i^k \sqrt{N}
\langle \gl_f(N)\prod_{j = 1}^nP(f;\phi_j) \right\rangle_\ast.
\end{equation}

In conclusion, if $\supp(\hphi) \subset (-1,1)$, we have
\begin{equation}\label{eq:full_grp_in_terms_of_S1}
\lim_{N\to\infty}
\left\langle\prod_{j = 1}^n\left(D(f;\phi_j)-\<D(f;\phi_j)\>_\ast\right)\right\rangle_\ast
\ =\ (-1)^n \lim_{N\to\infty} S_1^{(n)}
\end{equation}
and
\begin{equation}\label{eq:nth_mmt_in_terms_of_S}
\lim_{N\to\infty}
\left\langle\prod_{j = 1}^n\left(D(f;\phi_j)-\<D(f;\phi_j)\>_\pm\right)\right\rangle_\pm
\ =\ (-1)^n \lim_{N\to\infty} S_1^{(n)}\ \pm\ (-1)^{n+1}
\lim_{N\to\infty} S_2^{(n)}
\end{equation}
(assuming all limits exist, and \cite{HuM} show that these limits do exist), where
\begin{equation}\label{eq:S_1}
S_1^{(n)}\ :=\ \sum_{p_1 \notdiv N , \dots, p_n \notdiv N}
\prod_{j=1}^n \left( \phir{j} \pfrac{j} \right) \< \prod_{j=1}^n
\gl_f(p_i) \>_\ast
\end{equation}
and
\begin{equation}\label{eq:S_2}
S_2^{(n)}\ :=\ \i^k \sqrt{N} \sum_{p_1 \notdiv N , \dots, p_n
\notdiv N} \prod_{j=1}^n \left( \phir{j} \pfrac{j} \right) \<
\gl_f(N) \prod_{j=1}^n \gl_f(p_i) \>_\ast.
\end{equation}

\begin{rek}
    The equations above assume that ${\rm supp}(\hphi) \subset (-1,1)$ which does not affect our theorem since that assumes that ${\rm supp}(\hphi) \subset (-1/(n-1),1/(n-1))$ for the $n$\textsuperscript{{\rm th}} moment, and $(-1/(n-1),1/(n-1)) \subset (-1,1)$ for all positive integers $> 1$.
\end{rek}

\section{Proof of Theorem~\ref{thm: generalized centered moment}}\label{sec: proof of first theorem}

The starting point of our analysis is identical to that in \cite{HuM}, which is not surprising since our work generalizes their result from the case of all equal test functions to the case where all $n$ test functions are arbitrary. We therefore concentrate only on the places where new arguments are needed, specifically on the complications that arise when generalizing. We thus do not include detailed proof for results that fall immediately from arguing as in \cite{HuM}; instead, we discuss the changes we need to make to the arguments in \cite{HuM}, and recommend interested readers to sections $3$, $4$, and Appendix E of \cite{HuM}.

\subsection{Generalizing The First S Term}

\begin{lem} \label{lem:GeneralizedExpressionS1}
Let $S_1^{(n)}$ be defined as in \eqref{eq:S_1}, and assume GRH for $L(s,f)$. Then if $\supp(\hphi_j) \subset (-\frac{1}{n}
\frac{2k-1}{k}, \frac{1}{n} \frac{2k-1}{k})$,
\begin{equation}
\twocase{\lim_{\substack{N\to\infty \\ N\ {\rm prime}}} S_1^{(n)}
\ = \ }{\sum_{ k = 1}^{(2m-1)!!} \prod_{\{a_l,b_l\} \in \mathcal{M}_{2m}^{k} \atop l = 1}^{m} \sigma^{2}_{\phi_{a_l} \phi_{b_l}}}{if $n=2m$ is even}{0}{if $n$ is odd,}
\end{equation}
where
\begin{equation}
\sigma^{2}_{\phi_a \phi_b} \ := \ 2\int_{-\infty}^{\infty} |y| \hphi_a(y)\hphi_b(y) \;\d y.
\end{equation}
Note that $a, b$ are distinct integers from the set $\{1, 2, \dots, 2m\}$.
\end{lem}

\begin{proof}
Similar to the approach of \cite{HuM}, we first split the sum over primes into sums over powers of distinct primes: $p_1\cdots p_n = q_1^{n_1} \cdots q_\ell^{n_\ell}$
with the $q_j$ distinct. Thus, we obtain
\begin{equation} \label{eq: lamda term}
\prod_{j=1}^n \gl_f(p_i)\ =\ \prod_{j=1}^\ell \gl_f(q_j)^{n_j}.
\end{equation}

By the multiplicativity of the $\lambda_f$, proven in Lemma (2.8) of \cite{HuM}, we can rewrite $\gl_f(q_j)^{n_j}$ as a sum of $\gl_f(q_j^{m_j})$. In this sum we have $m_j \le n_j$ and $m_j$ and $n_j$ are non-negative integers of the same parity (i.e., either both even or both odd).

For $\prod_{i=1}^n \gl_f(p_i)$ to have a constant term (i.e., $\gl_f(1)$), $p_1\cdots p_n$ must equal a perfect square. We call this constant term the main term. We can only get a main term  when $n=2m$ is an even integer because only here can we pair up primes (as they may potentially occur for an even number of times). We first consider the case when each $n_j = 2$ so that each prime occurs exactly twice. Since each $\phir{}$ may contribute differently (as the test functions are not necessarily the same), we need to separate their pairings into cases and evaluate the sum separately.

To illustrate the method, rather than doing the general case, we determine the different cases and sums of the main term for $S_1^{(4)}$ below. This calculation easily generalizes to all even $n=2m$.
\begin{align}\label{eq:S_1^{(4)}}
    S_1^{(4)} = \sum_{p_1 \notdiv N , \dots, p_4 \notdiv N}\prod_{j=1}^4 \left( \phir{j} \pfrac{j} \right) \< \prod_{j=1}^4
\gl_f(p_j) \>_\ast
\end{align}

Note that we still assume each $n_j = 2$. Thus, to pair terms up, we have the following three cases. \\

\noindent \textbf{Case 1. $p_1 = p_2, \ p_3 = p_4$:}
\footnotesize
\begin{align}\label{eq: S_1^4case1}
    \lim_{\substack{N\to\infty \\ N\ {\rm prime}}} & \left(\sum_{p_1\notdiv N}  \hphi_1\left(\frac{\log{p_1}}{\log R}\right)\hphi_2\left(\frac{\log p_1}{\log R}\right)\left(\frac{2\log{p_1}}{\sqrt{p_1}\log{R}}\right)^2 \right) \times \left(\sum_{p_3\notdiv N} \hphi_3\left(\frac{\log{p_3}}{\log R}\right)\hphi_4\left(\frac{\log p_3}{\log R}\right)\left(\frac{2\log{p_3}}{\sqrt{p_3}\log{R}}\right)^2 \right).
\end{align}
\normalsize
\ \\

\noindent \textbf{Case 2. $p_1 = p_3, \ p_2 = p_4$:}
\footnotesize
\begin{align}\label{eq: S_1^4case2}
 \lim_{\substack{N\to\infty \\ N\ {\rm prime}}} & \left(\sum_{p_1\notdiv N} \hphi_1\left(\frac{\log{p_1}}{\log R}\right)\hphi_3\left(\frac{\log p_1}{\log R}\right)\left(\frac{2\log{p_1}}{\sqrt{p_1}\log{R}}\right)^2 \right) \times  \left(\sum_{p_2\notdiv N} \hphi_2\left(\frac{\log{p_2}}{\log R}\right)\hphi_4\left(\frac{\log p_2}{\log R}\right)\left(\frac{2\log{p_2}}{\sqrt{p_2}\log{R}}\right)^2 \right).
\end{align}
\normalsize
\ \\

\noindent \textbf{Case 3. $p_1 = p_4, \ p_2 = p_3$:}
\footnotesize
\begin{align}\label{eq: S_1^4case3}
 \lim_{\substack{N\to\infty \\ N\ {\rm prime}}} &  \left(\sum_{p_1\notdiv N} \hphi_1\left(\frac{\log{p_1}}{\log R}\right)\hphi_4\left(\frac{\log p_1}{\log R}\right)\left(\frac{2\log{p_1}}{\sqrt{p_1}\log{R}}\right)^2 \right)  \times \left(\sum_{p_2\notdiv N} \hphi_2\left(\frac{\log{p_2}}{\log R}\right)\hphi_3\left(\frac{\log p_2}{\log R}\right)\left(\frac{2\log{p_2}}{\sqrt{p_2}\log{R}}\right)^2 \right).
\end{align}
\normalsize

\ \\

Using the Prime Number Theorem (to first order the number of primes up to $x$ is asymptotic to $x/\log x$), we can evaluate the prime sums, and using the evenness of $\hphi_j$, the contribution from each term is
\begin{align}\label{eq:2mchoosempairs}
     \lim_{\substack{N\to\infty \\ N\ {\rm prime}}} \left(\sum_{p_k\notdiv N} \hphi_i\left(\frac{\log{p_k}}{\log R}\right)\hphi_j\left(\frac{\log p_k}{\log R}\right)\left(\frac{2\log{p_k}}{\sqrt{p_k}\log{R}}\right)^2 \right)\ & =\ 2\int_{-\infty}^\infty \hphi_i(y)\hphi_j(y) |y|\; \d y 
     \nonumber\\
     & = \ \sigma_{\phi_i\phi_j}^2.     
\end{align}

\begin{rek}
    The notation in \eqref{eq:2mchoosempairs} is a little different than \emph{\cite{HuM}}, where what they call $\sigma_{\phi}^2$ we call $\sigma_{\phi\phi}^2$. The reason is they only considered the case where all the test functions were the same, so each pair was always $\phi$ with $\phi$; our case is more general, and the notation must reflect which test functions are involved. We emphasize here that the ordering that the test function appear in the subscript does not matter, but we label them from small to large for consistency.
\end{rek}

The integral is the variance $\sigma^2_{\phi_i,\phi_j}$ because of the support condition on $\hphi_i$ and $\hphi_j$.

Using \eqref{eq:2mchoosempairs}, the sum of the three main terms of $S_1^{(4)}$ can be simplified into
\begin{equation}\label{eq: simplifiedS_1^4}
    \sigma_{\phi_1\phi_2}^2\sigma_{\phi_3\phi_4}^2+\sigma_{\phi_1\phi_3}^2\sigma_{\phi_2\phi_4}^2+\sigma_{\phi_1\phi_4}^2\sigma_{\phi_2\phi_3}^2.
\end{equation}

Note that if $\phi = \phi_1 = \phi_2 = \phi_3 = \phi_4$, \eqref{eq: simplifiedS_1^4} is just $3\sigma_{\phi\phi}^2$, which agrees with the equation in \cite{HuM}. We can generalize \eqref{eq: simplifiedS_1^4} for all main terms of even $n = 2m$:
\begin{equation}\label{eq: generalsumS_1}
    \sum_{k = 1}^{(2m-1)!!} \prod_{\substack{\{a_j,b_j\} \in \mathcal{M}_{2m}^k \\ j = 1}}^{m} \sigma^{2}_{\phi_{a_j} \phi_{b_j}}.
\end{equation}

Aside from the main term, the sub-terms would have some $n_j \geq 4$ in \eqref{eq: lamda term}. In this case, we obtain similar formula as \eqref{eq:2mchoosempairs}, but with different possibilities for pairing . The sums we have will look similar to the following, assuming $n_j = c_1+c_2+\cdots+c_n, \ c_j \in \mathbb{N}$:
\begin{equation} \label{eq: other n_j terms}
    \sum_{p\notdiv N} \hphi_1\left(\frac{2\log p}{\sqrt{p} \log R}\right)^{c_1}\hphi_2\left(\frac{2\log p}{\sqrt{p} \log R}\right)^{c_2}\cdots\hphi_n\left(\frac{2\log p}{\sqrt{p} \log R}\right)^{c_n}
    \left(\pfrac{}\right)^{n_j}.
\end{equation}
By the Prime Number Theorem, when $n_j = 2$, \eqref{eq: other n_j terms} is $O(1)$; however, when $n_j \ge 4$, \eqref{eq: other n_j terms} is $O\left(\log^{-4} R\right)$. Therefore, the contribution where at least one $n_j \ge 4$ is negligible.

The other contributions that arise from expanding $\prod_{i=1}^n \gl_f(p_i)$
are of the form
\footnotesize
\begin{align}
\sum_{\substack{q_1 \notdiv N , \dots, q_\ell \notdiv N \\ q_j
{\rm distinct}}} \prod_{j=1}^\ell & \ \hphi_1\left(\frac{2\log p}{\sqrt{p} \log R}\right)^{c_1}\hphi_2\left(\frac{2\log p}{\sqrt{p} \log R}\right)^{c_2}\cdots\hphi_n\left(\frac{2\log p}{\sqrt{p} \log R}\right)^{c_n} \times \left(\frac{2 \log q_j}{\sqrt{q_j} \log R}
\right)^{n_j} \< \gl_f(q_1^{m_1} \cdots q_\ell^{m_\ell}) \>_\ast.
\end{align}
\normalsize
To show that the other contributions from expanding $\prod_{i=1}^n \gl_f(p_i)$ is negligible, we use the fact that $\hphi_j$ is supported in an open interval $(-\frac{1}{n} \frac{2k-1}{k}, \frac{1}{n} \frac{2k-1}{k})$ to construct a new test function $\hpsi$ that is pointwise larger than each individual $|\hphi_i|$. We show that such a function exists, and our bound on this error term (which is larger than the actual error) follows directly  from \cite{HuM} if we replace each individual $\hphi_i$'s with $\hpsi$ and calculate the contribution following the steps laid out there. We thus obtain an upper bound for the actual contributions which is small enough to show that these terms are negligible. See Sections 3, 4, and Appendix E of \cite{HuM}.

By Lemma \ref{lem: functionBound},
\begin{equation}\label{eq: psiUnmodified}
    1+\sum_{j = 1}^n \hphi_j(x)^2 \ >\ \lvert \hphi_j(x) \rvert \ \text{for all } j.
\end{equation}

Although later analysis doesn't require $\hpsi$ to be Schwartz function, we still need to modify \eqref{eq: psiUnmodified} to ensure that it is infinitely differentiable and has support in $(-\frac{1}{n}\frac{2k-1}{k}, \frac{1}{n} \frac{2k-1}{k})$ like all $\hphi_i$. We have to bound the contribution from other terms, and the analysis in \cite{HuM} showed how to obtain an explicit formula for these other terms. However, for us, it doesn't matter how we get to such an explicit formula; we just need to bound these quantities

If $\supp(\hphi_j) = (-\omega,\omega)$, we define a useful function $\eta_\epsilon(x)$. We will multiply the function on the left hand side of \eqref{eq: psiUnmodified} by this function to obtain the needed properties. In particular, we will have a Schwartz function which is at least 1 and larger than all the individual $\hphi_i$'s where they are supported, and decays rapidly to zero. Consider for some very small $\epsilon > 0$, chosen so that the support of each $\hphi$ is contained in $(-\omega + \epsilon, \omega - \epsilon)$:
\begin{equation}\label{eq: eta(x)}
    \eta_\epsilon(x)\ :=\ \begin{cases}
    e^{\left(1 / (\omega^2 - (\omega - \epsilon)^2)\right)} \times
    e^{-\left(1 / (\omega^2 - x^2)\right)}
    , & \text{for } \lvert x \rvert < \omega\\
        \text{0}, & \text{for } \lvert x \rvert \geq \omega.
        \end{cases}
\end{equation}

Note that for every $0<\epsilon<\omega$, $1 \leq \eta_\epsilon(x)$ for every $x \in (-\omega+\epsilon, \omega-\epsilon)$. Thus, the product of \eqref{eq: psiUnmodified} and \eqref{eq: eta(x)} is still point-wise larger than any $\lvert \hphi_j(x) \rvert$. At the same time, $\eta_\epsilon(x)$ vanishes at $(-\omega, \omega)$ and is infinitely differentiable, so the product would also have appropriate support and is infinitely differentiable.

\begin{rek}
    Since the function is decreasing and positive
at the endpoints of the subinterval, clearly we can multiply by a sufficiently large $A$ to make it at least 1 in that subinterval. In fact the smallest $A$ that works is 1 divided by the value at the end of the subinterval, as our function is decreasing as $|x|$ increases, which is where our choice came from.
\end{rek}

Since $\eta_\epsilon(x)$ and the expression of \eqref{eq: psiUnmodified} are both infinitely differentiable, their product is  infinitely differentiable with support  $(-\omega, \omega)$ since $\supp(\eta_\epsilon(x)) = (-\omega, \omega)$. In conclusion, the product satisfies all our required properties. Thus, we define the new test function $\hpsi_{\epsilon}(x)$ as
\begin{equation} \label{eq: psi final}
    \hpsi_\epsilon(x)\ :=\ \eta_\epsilon(x)\times\left(1+\sum_{j = 1}^n \hphi_j(x)^2\right).
\end{equation}

To finish the proof, argue as in \cite{HuM} after replacing every $\hphi_j(x)$ with $\hpsi_\epsilon(x)$. All quantities are larger, but these new functions are admissible as they are infinitely differentiable and have compact support, which are the only properties that analysis assumes.
\end{proof}

\begin{rek}
    Theorem E.1 of \cite{HuM} extends the support in Lemma \ref{lem:GeneralizedExpressionS1} to $(-\frac{2}{n}, \frac{2}{n})$ for $2k \ge n$ by assuming GRH for Dirichlet $L$-functions. Same as in \cite{HuM}, support of $(-\frac{1}{n} \frac{2k-1}{k}, \frac{1}{n} \frac{2k-1}{k})$ is sufficient for $n \geq 3$, $k \geq 2$ for our purposes because to prove Theorem~\ref{thm: generalized centered moment} we need to handle support up to $\frac{1}{n-1}$, and $\frac{1}{n-1} \leq \frac{1}{n}\frac{2k-1}{k}$ in $n$ and $k$ satisfy the conditions mentioned earlier. Therefore, Lemma~\ref{lem:GeneralizedExpressionS1} evaluates $S_1^{(n)}$ for $\sigma < \frac1{n-1}$. If $n=2$, however, then $\frac1{n-1} > \frac{1}{n}\frac{2k-1}{k}$, we will need the generalized Theorem E.1 of \cite{HuM}. However, since this result follows almost directly by just replacing each $\phi$ by distinct $\phi_j$ from the proof in \cite{HuM}, we will not expand the proof here. Instead, we refer the readers to Appendix E of \cite{HuM} if interested.
\end{rek}

\subsection{Generalizing The Second S Term}
All that remains to prove Theorem \ref{thm: generalized centered moment} is
to show that if $\sigma<\frac{1}{n-1}$ then
\begin{align}
    \lim_{\substack{N\to\infty \\ N\ {\rm prime}}} S_2^{(n)}\ &=\
     (-1)^{n-1} 2^{n-1}\left[ \int_{-\infty}^\infty
    \phi_1(x)\cdots \phi_n(x)
    \frac{\sin 2\pi x}{2\pi x} \;\d x \right.  \left.  - \foh \phi_1(0)\cdots \phi_n(0) \right].
\end{align}

We highlight the changes needed to generalize the argument in \cite{HuM}. Instead of defining of $\Phi(x)$ as $\phi(x)^n$, we now define it as $\phi_1(x)\cdots\phi_n(x)$ and the result falls directly from this replacement if one follows the steps laid out in \cite{HuM}. To bound the error term, we use the same technique of what we did for bounding the error term of $S_1$ -- replacing every $\phi_j$ with $\psi_{\epsilon}(x) = (1+\prod_{j=1}^n \phi_j(x))\times\eta_{\epsilon}(x)$. This does not affect the proof as $\psi_{\epsilon}(x)$ has all of our desired qualities.

Combining the two generalized terms $S_1^{(n)}$ and $S_2^{(n)}$, we obtain Theorem~\ref{thm: generalized centered moment}. \hfill $\Box$


\section{Proof of Theorem~\ref{thm: generalized less support}}\label{sec: proof for second theorem (less support)}

Since we handled the $S_1^{(n)}$ term in Lemma \ref{lem:GeneralizedExpressionS1}, all that we need to show is that the contribution from the $S_2^{(n)}$ term is negligible for $\supp(\hphi) \subset(-\frac{1}{n}\frac{2k-1}{k} ,\frac{1}{n}\frac{2k-1}{k})$. This error analysis is very similar to the argument for $S_1^{(n)}$. We can replace every $\hphi_j$ with $\hpsi_{\epsilon}$ as defined in \eqref{eq: psi final} and follow the steps of Lemma 3.2 of \cite{HuM}.

\section{Bounding Order of Vanishing}\label{sec: bounds order of vanishing}

\begin{defi}[Probability function]
    We define $p_j(N)$ as the percent of cuspidal newforms in the family with square-free level $N$ that have exactly $j$ zeros at the central point.
\end{defi}

Note trivially $\sum_{j \ge 0} p_j(N) = 1$. We frequently bound sums such as $\sum_{j \ge r} p_j(N)$, the percent that have rank at least $r$. If our family is SO(even) then only terms with $j$ even can be non-zero (and analogously for SO(odd)).

\begin{rek} Typically after the first few initial papers, subsequent progress is small, usually only slight improvements in later digits. We see this in the following tables in the improvement in bounding the percentage of forms that vanish to order exactly 5; our result is .0658044, improving but only slightly the earlier result of 0.06744290. However, the value of our approach becomes apparent as we look at higher and higher vanishing. Our result is better by about a factor of 2 for vanishing to degree 6, improving the previous bound of 0.0157687 to 0.00853841, and for vanishing of exact order 7 we go from 0.0299746 to 0.00221997, which is more than an order of magnitude better. The results are even stronger for higher vanishing. \end{rek}

\subsection{Bounds from Densities}

Often previous works show how to obtain bounds but do not actually enumerate them. Thus, in order to compare the two methods we have developed to obtain bound with what was known, we made the formulas explicit for bounds arising from the $1$- and $2$-level densities below, and calculated the bounds for various choices of test functions (using the optimal test functions when known).

For the $1$-level density,
\begin{equation}
	\sum_{m = 1}^\infty m p_m(N) \ \leq \ \frac{1}{\phi(0)} \int_\R \phi(x) W_{1, G} (x) dx dy \label{eq:boundorder1} \ =\  \mathbb{E}[(D_1 (f; \phi), Q)].
\end{equation}

We apply \eqref{eq:boundorder1} and obtain the following equations for bounds from $1$-level density:
\begin{align}
    p_{2m}(N)\ \leq\ \frac{\mathbb{E}[(D_1 (f; \phi), Q)]}{2m} \qquad &\text{if } G = \SpOrthE,\nonumber\\
    p_{2m+1}(N)\ \leq\  \frac{\mathbb{E}[(D_1 (f; \phi), Q)]}{2m+1} \qquad &\text{if } G = \SpOrthO.
\end{align}

For the $2$-level densities:
\begin{align}
    & \sum_{m = 1}^\infty \left(2m(2m - 2) p_{2m}(N) + (2m)^2 p_{2m+1}(N) \right) \ \leq \ \frac{1}{\Phi(0,0)} \int_{\R^2} \Phi(x, y) W_{2, G} (x, y) dx dy. \label{eq:boundorder2}
\end{align}

We apply \eqref{eq:boundorder2} and obtain the following equations for bounds from $2$-level density,
\begin{align}
    p_{2m}(N)\ \leq\  \frac{\mathbb{E}[(D_2 (f; \phi), Q)]}{2m(2m-2)} \qquad &\text{if } G = \SpOrthE,\nonumber\\
    p_{2m+1}(N)\ \leq\ \frac{\mathbb{E}[(D_2 (f; \phi), Q)]}{2m(2m)} \qquad &\text{if } G = \SpOrthO.
\end{align}


\subsection{Bounds from Centered Moments}\label{subsec: bounds from centered moments sum}
Consider the term
\begin{center}
    $D(f;\phi)-\<D(f;\phi)\>_\sigma$.
\end{center}
If $L(s, f)$ has $j$ zeros at the central point, we can throw away the contribution from all other zeros and obtain the following (by Theorem 1.1 of \cite{HuM}):
\begin{equation}\label{eq: boundCoefficients}
        D(f;\phi)-\<D(f;\phi)\>_\sigma\ \geq\ j\phi(0)-\<D(f;\phi)\>_\sigma \\
    \ = \ j\phi(0)-\left(\hphi(0)+\frac{1}{2}\phi(0)\right).
\end{equation}

\begin{rek}\label{rek: limits on bounds}
    For our application, we need the right-hand side of \eqref{eq: boundCoefficients} to be positive because when we raise terms to higher even powers later to compute the bounds from higher even centered moments, we may violate the inequality. To illustrate this potential violation, consider for instance that for a particular choice of $f$, $\<D(f;\phi)\>_\sigma = 1$, $D(f;\phi) = 1.5$, and the contribution from the zeros at the central point is $0.25$. When we evaluate the second centered moment, if we only consider the contribution of the zeros and square, we end up with $0.5625$; this is larger than the $0.25$ we obtain if we keep everything and use $D(f; \phi)$. This is an issue because the term is larger than what the upper bound, indicating of course that we cannot take this as an upper bound as squaring the negative sign reversed the direction of the inequality. Therefore, there are restrictions on the order of vanishing bounds we can obtain for each test function $\phi$. Specifically, the order of vanishing must be greater than $(\hphi(0)+\phi(0)/2)/(\phi(0))$. We denote the smallest order of vanishing that we can evaluate by $c_{\phi}$.
\end{rek}

\begin{defi}[Smallest Order of Vanishing]
    We denote by $c_{\phi_j}$ the smallest order of vanishing that can be evaluated using the even centered moment for a given test function $\phi_j$. Specifically, $c_{\phi_j}$ is the smallest integer that satisfies the following inequality
    \begin{equation}
        c_{\phi_j}\  >\  \frac{\hphi_j(0)+\frac{1}{2}\phi_j(0)}{\phi_j(0)}.
    \end{equation}
\end{defi}

\begin{rek}
     For our naive test function $\varphi_{\rm naive}$, $c_{\varphi_{\rm naive}} = 5 > (\hvarphi_{\rm naive}(0)+\varphi_{\rm naive}(0)/2)/\varphi_{\rm naive}(0) = 4.5$. For our better test function $\phi$ generated by $g = \sin^2(x)$ for $|x| > 1/8$, $c_{\phi} = 8 > (\hphi(0)+\phi(0)/2)/\phi(0) = 7.69993$.
\end{rek}

Combining Remark \ref{rek: limits on bounds} and \eqref{eq: boundCoefficients}, we obtain the following for even moment $n=2m$ and $f\in H_k^\sigma(N)$ where all test functions $\phi$ are the same, and we assume $r \geq c_{\phi}$:
\begin{align}
    & \left\langle\left(  D(f;\phi)-\<D(f;\phi)\>_\sigma\right)^{2m}\right\rangle_\sigma \ = \ \sum_{\substack{j=0 \\ f \text{ such that } L(s,f) \text{ has } j \\ \text{ order of vanishing }}}^{\infty} \frac{1}{|H_k^{\sigma}(N)|}\left(D(f;\phi)-\<D(f;\phi)\>_\sigma\right)^{2m} \nonumber\\
    &\geq \sum_{\substack{j=r \\ f \text{ such that } L(s,f) \text{ has } j \\ \text{ order of vanishing }}}^{\infty} \frac{1}{|H_k^{\sigma}(N)|}\left(r\phi(0)-\left(\hphi(0)+\frac{1}{2}\phi(0)\right)\right)^{2m}. \label{eq: finalBoundSum no p_j}
\end{align}


Since we group together the $f$ with a certain order of vanishing and divide by the cardinality of the family, the sum over $f$ divided by the family size becomes $p_j(N)$ times the contribution. Thus, we can simplify \eqref{eq: finalBoundSum no p_j} to
\begin{align}
    & \left\langle\left(  D(f;\phi)-\<D(f;\phi)\>_\sigma\right)^{2m}\right\rangle_\sigma
    \ \geq \ \sum_{j=r}^{\infty} p_j(N)\left(r\phi(0)-\left(\hphi(0)+\frac{1}{2}\phi(0)\right)\right)^{2m}. \label{eq: finalBoundSum}
\end{align}

\begin{rek}\label{rek: evenMomentBetter}
    We use \eqref{eq: finalBoundSum} to obtain bounds for $p_j(N)$ since $$\left\langle\left(D(f;\phi)-\<D(f;\phi)\>_\sigma\right)^{2m}\right\rangle_\sigma \ \ {\rm and}\ \  \left(r\phi(0)-\left(\hphi(0)+\frac{1}{2}\phi(0)\right)\right)$$  are both easily computable for a given pair $\phi, \ \hphi$. We emphasize here that \eqref{eq: finalBoundSum} only holds for even centered moments because the contribution from every $f$ is always positive in this case as every $D(f;\phi)-\<D(f;\phi)\>_\sigma$ is raised to an even power. Thus, when we drop contributions from $f$ that has fewer than $r$ zeros at the central point, we won't drop negative terms, so the inequality preserves. This is the main reason why it is better to study even moments than odd moments.
\end{rek}

\begin{rek}\label{rek: generalizeBoundSum}
    We use equation \eqref{eq: finalBoundSum} in the special case where all test functions $\phi$ are identical. We note here that we may run into issues if we work with the most generalized case where we treat all $\phi_1, \phi_2, \dots, \phi_{2m}$ as distinct because when we drop contributions from $f$ that has fewer than $r$ zeros at the central point, we may be dropping negative terms, so the inequality will not hold for certain. A solution to this problem is to make sure all test functions occur for an even number of times. In other words, we treat all $\phi_1, \phi_2, \dots, \phi_{m}$ as distinct when using the $2m$\textsuperscript{{\rm th}} centered moment to obtain upper bounds for order of vanishing as shown below.
\end{rek}

For even moment $n=2m$ and $f\in H_k^\sigma(N)$, we have the following more generalized equation, and again we assume $r \geq \text{max}(c_{\phi_1}\dots c_{\phi_m}$):
\begin{align}
    & \left\langle\prod_{s=1}^{m}\left( D(f;\phi_s)-\<D(f;\phi_s)\>_\sigma\right)^{2}\right\rangle_\sigma\ =\ \sum_{\substack{j=0 \\ f \text{ such that } L(s,f) \text{ has } j \\ \text{ order of vanishing }}}^{\infty} \frac{1}{|H_k^{\sigma}(N)|}\left(\prod_{s=1}^m D(f;\phi_s)-\<D(f;\phi_s)\>_\sigma\right)^{2} \nonumber\\
    &\geq \sum_{\substack{j=r \\ f \text{ such that } L(s,f) \text{ has } j \\ \text{ order of vanishing }}}^{\infty}  \frac{1}{|H_k^{\sigma}(N)|}\left(\prod_{s=1}^m r\phi_s(0)-\left(\hphi_s(0)+\frac{1}{2}\phi_s(0)\right)\right)^{2}. \label{eq: momentInequalityGeneralized no p_j}
\end{align}


Similarly, we can apply the same technique that we applied to \eqref{eq: finalBoundSum no p_j} to \eqref{eq: momentInequalityGeneralized no p_j} and obtain 
\begin{align}
    & \left\langle\prod_{s=1}^{m}\left( D(f;\phi_s)-\<D(f;\phi_s)\>_\sigma\right)^{2}\right\rangle_\sigma  \ \geq \ \sum_{j=r}^{\infty}  p_j(N)\left(\prod_{s=1}^m r\phi_s(0)-\left(\hphi_s(0)+\frac{1}{2}\phi_s(0)\right)\right)^{2}.
\end{align}

We explicitly state the upper bounds collected from these equations in the below section.

\section{Table of Numerical Bounds}\label{sec: table of numerical bounds}

\begin{rek}
    The coefficient $2m(2m - 2)$ for even order of vanishing $2m$ comes from the fact that $\gamma_f^{(j_i)} \neq \pm \gamma_f^{(j_j)}$. Thus, after we pick the first zero from the $2m$ zeros, we only have $2m-2$ choices for the second zeros. The coefficient $(2m)^2$ for odd order of vanishing $2m+1$ comes from the fact that $\gamma_f^{(j_i)} \neq \pm \gamma_f^{(j_j)}$ and also the fact that odd order of vanishing is only possible for SO(odd), which has an extra zero $\gamma_f^{(j_0)}$ at the central point; this extra zero needs to be considered separately. There are two cases: case $1$, we don't pick the extra zero as the first zero, and we obtain the following $(2m)(2m-1)$; case $2$, we pick the extra zero as the first zero, and we obtain the following $(1)(2m)$. Summing these mutually independent cases, we obtain our desired $(2m)^2$ for the coefficient for odd vanishing.
\end{rek}
\normalsize

\subsection{SO(even) Bounds}
We present calculated SO(even) bounds below.
\begin{table}[H]
    \centering
    \begin{tabular}{|l|l|l|l|}
    \hline
        Order vanish & $1$-level & $2$-level & $4$\textsuperscript{th} centered moment\textsuperscript{*}\\ \hline
        6  & 0.144090 & 0.01576870  & 0.00853841 \\ \hline
        8  & 0.108067 & 0.00788434  & 0.00081336 \\ \hline
        10 & 0.086454 & 0.00473060  & 0.00018684 \\ \hline
        20 & 0.043227 & 0.00105125  & 4.49988$\cdot 10^{-6}$ \\ \hline
        50 & 0.017290 & 0.00015768 & 7.13387$\cdot 10^{-8}$ \\ \hline
    \end{tabular} \small
\caption{Comparison of order of vanishing bounds from various approaches. \\
These are upper bounds for vanishing at least $r$ (number in order vanishing column).\\
For the $1$-level column, we calculated the bound using the optimal $1$-level bound from \cite{ILS}. The support of the Fourier transform of the test function used is $(-2,2)$.\\
For the $2$-level column, we calculated the bound using the optimal $2$-level bound from \cite{BCDMZ}. The support of the Fourier transform of the test functions used is $(-1,1)$.\\
For the $4$\textsuperscript{th} centered moment\textsuperscript{*} column, the \textsuperscript{*} signifies that we used the $4$ copies of the naive test functions $\varphi_{\rm naive}$. The support of the Fourier transform of the test function used is $(-1/3,1/3)$.}
\end{table} \normalsize

\normalsize
Appendix A of \cite{ILS} discusses a general result of functional analysis, which is the starting point of the efforts to determine the optimal test function. Gallagher \cite{Ga} proved that the candidate test functions $\phi$ for the optimal test function must be of the form $\phi(z)=|h(z)|^2$, where $h$ is an entire function of exponential type $1$ and $h \in L^2(R)$ (i.e., is square-integrable). One property of Fourier transforms is that the Fourier transform of the convolution of two functions is the product of the Fourier transforms. Using this property, we can rewrite this condition for $\phi$ in terms of $\hphi$:
\begin{equation}
    \hphi(x) \ = \ (g \ast \Tilde{g})(x),
\end{equation}
where $\ast$ represents convolution and
\begin{equation}
    \Tilde{g}(x)\ = \ \overline{g(-x)}.
\end{equation}

\begin{rek}
    We choose to vary $g$ to find a better set of test functions instead of varying $\phi$ and $\phihat$ because the restrictions that $\phi$ must be non-negative and even, and $\phihat$ compactly supported, limit the closed form choices for $\phi$ and $\phihat$ significantly; in other words, it is hard to write down such candidate functions. On the other hand, if we choose $g$ to be even, non-negative and compactly supported, $\phihat$ and $\phi$ will also be even and non-negative, and if the support of $g$ is $(-\omega/2, \ \omega/2)$, the support of $\phihat$ will be $(-\omega, \ \omega)$.
\end{rek}

After testing out various $g$ for the $4$\textsuperscript{th} centered moment, we find that the combination of a pair of $\phi(x)$ generated by $g(x) = \sin(x^2)$ for $|x|<1/8$ and a pair of the naive test function $\varphi_{\rm naive}(x)$ (defined in \ref{def: naive test function}) supported from $(-1/4, \ 1/4)$ yield better results than 4 copies of the naive test function when $n \geq 100$. We emphasize here that this choice of $g$ is most definitely not optimal, but it is nevertheless an improvement from the naive test functions. In particular, this shows that for very high vanishing, the flexibility of using higher moments yields better bounds.

We used Theorem \ref{thm: generalized centered moment} to calculate the naive bounds and Theorem \ref{thm: generalized less support} to calculate the bounds for better test functions because they have different supports.

\begin{table}[H]
    \centering
    \begin{tabular}{|l|l|l|l|l|}
    \hline
                Order & \ & \ & 4th centered & 4th centered \\
        vanishing & 1-level & 2-level & moment\textsuperscript{*} & moment\textsuperscript{**} \\
        \hline
        100  & 0.0086454 & 3.86172$\cdot 10^{-5}$ & 3.84617$\cdot 10^{-9}$  & 3.7858$\cdot 10^{-9}$ \\ \hline
        200  & 0.0043227 & 9.55677$\cdot 10^{-6}$ & 2.23711$\cdot 10^{-10}$ & 2.0812$\cdot 10^{-10}$ \\ \hline
        300  & 0.0028818 & 4.23320$\cdot 10^{-6}$ & 4.31557$\cdot 10^{-11}$ & 3.9427$\cdot 10^{-11}$ \\ \hline
        800  & 0.0010806 & 5.92807$\cdot 10^{-7}$ & 8.28694$\cdot 10^{-13}$ & 7.4047$\cdot 10^{-13}$ \\ \hline
        900  & 0.0009606 & 4.68261$\cdot 10^{-7}$ & 5.16340$\cdot 10^{-13}$ & 4.6069$\cdot 10^{-13}$ \\ \hline
        1000 & 0.0008645 & 3.79207$\cdot 10^{-7}$ & 3.38242$\cdot 10^{-13}$ & 3.0144$\cdot 10^{-13}$ \\ \hline
             &           &                   &                    &  \\ \hline
        2020 & 0.0004279 & 9.28398$\cdot 10^{-8}$ & 2.01718$\cdot 10^{-14}$ & 1.7882$\cdot 10^{-14}$ \\ \hline
    \end{tabular} \small
    \caption{Comparison of order of vanishing bounds from various approaches. \\
    These are upper bounds for vanishing at least $r$ (number in order vanishing column). \\
    For the $1$-level column, we calculated the bound using the optimal $1$-level bound from \cite{ILS}. The support of the Fourier transform of the test functions used is $(-2,2)$.\\
    For the $2$-level column, we calculated the bound using the optimal $2$-level bound from \cite{BCDMZ}. The support of the Fourier transform of the test functions used is $(-1,1)$.\\
    For the $4$\textsuperscript{th} centered moment\textsuperscript{*} column, the \textsuperscript{*} signifies that we used the $4$ copies of the naive test functions $\varphi_{\rm naive}$. The support of the Fourier transform of the test functions used is $(-1/3,1/3)$.\\
    For the $4$\textsuperscript{th} centered moment\textsuperscript{**} column, the \textsuperscript{**} signifies that we used $2$ copies of $\phi$ generated by $g(x)=\sin{x^2}$ for $x<1/8$ and $2$ copies of the naive test functions $\varphi_{\rm naive}$. The support of the Fourier transform of the test functions used is $(-1/4,1/4)$.
    }
\end{table} \normalsize
\normalsize
We also obtain improved bounds for the SO(odd) group, which is shown in the next section.
\subsection{SO(odd) Bounds} We present calculated SO(odd) bounds below.
\begin{table}[H]
    \centering
    \begin{tabular}{|l|l|l|l|}
    \hline
        Order vanishing & $1$-level & $2$-level & $4$\textsuperscript{th} centered moment\textsuperscript{*} \\ \hline
        5  & 0.22290 & 0.0674429 & 0.0658044 \\ \hline
        7  & 0.15922 & 0.0299746 & 0.0022199 \\ \hline
        9  & 0.12383 & 0.0168607 & 0.0003640 \\ \hline
        19 & 0.05866 & 0.0033305 & 5.77156$\cdot 10^{-6}$ \\ \hline
        49 & 0.02274 & 0.0004683 & 7.77275$\cdot 10^{-8}$ \\ \hline
    \end{tabular}\small
        \caption{Comparison of order of vanishing bounds from various approaches. \\
        These are upper bounds for vanishing at least $r$ (number in order vanishing column).\\
        For the $1$-level column, we calculated the bound using the optimal $1$-level bound from \cite{ILS}. The support of the Fourier transform of the test function used is $(-2,2)$.\\
        For the $2$-level column, we calculated the bound using the optimal $2$-level bound from \cite{BCDMZ}. The support of the Fourier transform of the test functions used is $(-1,1)$.\\
        For the $4$\textsuperscript{th} centered moment\textsuperscript{*} column, the \textsuperscript{*} signifies that we used the $4$ copies of the naive test functions $\varphi_{\rm naive}$. The support of the Fourier transform of the test function used is $(-1/3,1/3)$.}
\end{table} \normalsize

\begin{table}[H]
    \centering
    \begin{tabular}{|l|l|l|l|l|}
    \hline
        Order & \ & \ & 4th centered & 4th centered \\
        vanishing & 1-level & 2-level & moment\textsuperscript{*} & moment\textsuperscript{**} \\
        \hline
        99 & 0.011258 & 0.000112358 & 4.00504$\cdot 10^{-9}$ & 3.95151$\cdot 10^{-9}$ \\ \hline
        199 & 0.00560070 & 2.75249$\cdot 10^{-5}$ & 2.28052$\cdot 10^{-10}$ & 2.12471$\cdot 10^{-10}$ \\ \hline
        299 & 0.00372756 & 1.21513$\cdot 10^{-5}$ & 4.36908$\cdot 10^{-11}$ & 3.99684$\cdot 10^{-11}$ \\ \hline
        399 & 0.00279333 & 6.81224$\cdot 10^{-6}$ & 1.36155$\cdot 10^{-11}$ & 1.23440$\cdot 10^{-11}$ \\ \hline
        799 & 0.00139492 & 1.69454$\cdot 10^{-6}$ & 8.31878$\cdot 10^{-13}$ & 7.44215$\cdot 10^{-13}$ \\ \hline
        899 & 0.00123975 & 1.33815$\cdot 10^{-6}$ & 5.18033$\cdot 10^{-13}$ & 4.62765$\cdot 10^{-13}$ \\ \hline
        999 & 0.00111566 & 1.08342$\cdot 10^{-6}$ & 3.39199$\cdot 10^{-13}$ & 3.02656$\cdot 10^{-13}$ \\ \hline
         &  &  &  &  \\ \hline
        2021 & 0.000551479 & 2.64456$\cdot 10^{-7}$ & 2.01079$\cdot 10^{-14}$ & 1.78466$\cdot 10^{-14}$ \\ \hline
    \end{tabular} \small
    \caption{Comparison of order of vanishing bounds from various approaches. \\
    These are upper bounds for vanishing at least $r$ (number in order vanishing column). \\
    For the $1$-level column, we calculated the bound using the optimal $1$-level bound from \cite{ILS}. The support of the Fourier transform of the test functions used is $(-2,2)$.\\
    For the $2$-level column, we calculated the bound using the optimal $2$-level bound from \cite{BCDMZ}. The support of the Fourier transform of the test functions used is $(-1,1)$.\\
    For the $4$\textsuperscript{th} centered moment\textsuperscript{*} column, the \textsuperscript{*} signifies that we used the $4$ copies of the naive test functions $\varphi_{\rm naive}$. The support of the Fourier transform of the test functions used is $(-1/3,1/3)$.\\
    For the $4$\textsuperscript{th} centered moment\textsuperscript{**} column, the \textsuperscript{**} signifies that we used $2$ copies of $\phi$ generated by $g(x)=\sin{x^2}$ for $x<1/8$ and $2$ copies of the naive test functions $\varphi_{\rm naive}$. The support of the Fourier transform of the test functions used is $(-1/4,1/4)$.
    }
\end{table} \normalsize

\normalsize

\section{Future Work}

This paper is the first exploration using higher moments to obtain explicit bounds. A natural extension is to use the sixth and higher moments. As we go higher and higher, we have smaller and smaller support but we have better results for sufficiently high vanishing; the question is where is the cut-off where one moment gives better results than another.

One of the biggest advantages of the Hughes-Miller framework is that we can easily see how new terms arise as we increase support. Thus, future work could also focus on increasing the support of the test function by improving the number theory calculation, in particular how much results improve when we increase support for a fixed level or moment; see \cite{C--} for recent results along these lines.

In addition, we present results specifically for the orthogonal groups; our techniques are also applicable to the symplectic and unitary symmetry groups, and we can obtain similar results.

Finally one could also further optimize the test functions (i.e., obtain better test functions and subsequent improvement in the bounds) through more advanced functional analysis, similar to analysis shown in Appendix A of \cite{ILS} and \cite{BCDMZ}.


\bibliography{biblio}

\ \\ 
\end{document}